\documentclass[preprint, 10pt, english]{elsarticle}
\usepackage{amsmath}
\usepackage{latexsym, amssymb}
\usepackage{mathtools}
\usepackage{amsthm}
\usepackage{color}
\usepackage{txfonts} % PPP1
\usepackage{babel}
\usepackage[all]{xy}

\newtheorem{thm}{Theorem}[section] %the resolution could also be [subsection]

\newtheorem{cor}[thm]{Corollary}
\newtheorem{defn}[thm]{Definition}

\newtheorem{exmpl}[thm]{Example}
\newtheorem{lem}[thm]{Lemma}
\newtheorem{prop}[thm]{Proposition}

\newtheorem{rem}[thm]{Remark}

\newtheorem{ques}[thm]{Question}

\DeclareMathOperator{\mychar}{char} %

\def\ra{{\rightarrow}}

\newcommand\operA[2]{{\if!#2!\operatorname{#1}\else{\operatorname{#1}_{#2}^{\phantom{I}}}\fi}} % To be used within Bdefs. Usage: $\operA{N}{K/F}$ produces $N_{K/F}$; $\operA{N}{}$ produces $N$.
\newcommand\set[1]{\{#1\}}

\newcommand\ind{\operatorname{ind}}

\newcommand\eq[1]{{(\ref{#1})}}% \eqref{#1} %
\newcommand\Eq[1]{{Equation~(\ref{#1})}}% \eqref{#1} %
\newcommand\Lref[1]{{lemma~\ref{#1}}}%
\newcommand\Pref[1]{{Proposition~\ref{#1}}}%
\newcommand\Eref[1]{{Example~\ref{#1}}}%
\newcommand\Cref[1]{{Corollary~\ref{#1}}}%
\newcommand\Rref[1]{{Remark~\ref{#1}}}%
\newcommand\Tref[1]{{Theorem~\ref{#1}}}%
\newcommand\Dref[1]{{Definition~\ref{#1}}}%
\newcommand\Sref[1]{{Section~\ref{#1}}}%
\newcommand\Ssref[1]{{Subsection~\ref{#1}}}%

\newcommand\an[1]{{{#1}_{\operatorname{an}}}}

\newcommand\tensor[1][]{{\otimes_{#1}}}
\def\sub{\subseteq}

\newcommand\Brp[1][2]{{\,{}_{#1}\!\operatorname{Br}}}

\def\dim{{\operatorname{dim}}}

\def\Z{\mathbb{Z}}
\def\F{\mathbb{F}}

\newcommand\Norm[1][]{\operA{N}{#1}}

\newcommand\mul[1]{{#1^{\times}}} % The multiplicative group
\newcommand\dimcol[2]{{[{#1}\!:\!{#2}]}} % Produces nicely spaced [K:F]. Don't use in subscripts or superscripts -- there LaTeX manages by his own.
\newcommand{\Trace}[1][]{\if!#1!\operatorname{Tr}\else{\operatorname{Tr}_{#1}^{\phantom{I}}}\fi} % Usage: $\Tr[K/F](a)$.

\long\def\forget#1\forgotten{{}} %
\newcommand\suchthat{{\,:\ \,}}
\newcommand\subjectto{{\,|\ }}
\def\co{{\,{:}\,}}
\newcommand\lam{{\lambda}}

\def\({\left(}
\def\){\right)}

\newcommand\card[1]{{\left|#1\right|}}
\newcommand\isom{{\,\cong\,}}

\def\Ra{\Rightarrow}
\newcommand\Qf[1]{{\left<{#1}\right>}}              % Quadratic form
\newcommand\Pf[1]{{\left<\left<{#1}\right>\right>}} % Pfister form
\newcommand\MPf[1]{{\left<\left<{#1}\right]\right]}} % Pfister form in char=2
\newcommand\QA[2]{{(#2,#1]}} %--- mirror image of the standard notation
\newcommand\Alb{{\frak{a}}} % The Albert form
\def\HQ{{\varmathbb{H}}} % PPP2
\def\Ra{{\Rightarrow}}
\newcommand\db[1]{{(\!\!\;(#1)\!\!\;)}}
\newcommand\LAY[3][]{{\begin{array}{c}\mbox{#2} \if#1!{}\else{+}\fi \\ \mbox{#3}\end{array}}}

\newif\iffurther
\furtherfalse
\newcommand\ideal[1]{{\left<#1\right>}}
\newcommand\sg[1]{{\left<#1\right>}}

\journal{Communications in Algebra}

\begin{document}
\begin{frontmatter}

\title{Linkage of Quadratic Pfister Forms}

\author{Adam Chapman}
\ead{adam1chapman@yahoo.com}
\address{Department of Computer Science, Tel-Hai Academic College, Upper Galilee, 12208 Israel}
\author{Shira Gilat}
\ead{shira.gilat@live.biu.ac.il}
\author{Uzi Vishne}
\ead{vishne@math.biu.ac.il}
\address{Department of Mathematics, Bar-Ilan University, Ramat-Gan 52900, Israel}
%\thanks{Vishne was partially supported by an Israeli Science Foundation (grant no. 2128/16)}

\begin{abstract}
We study the necessary conditions for sets of quadratic $n$-fold Pfister forms to have a common $(n-1)$-fold Pfister factor. For any set $S$ of $n$-fold Pfister forms generating a subgroup of $I_q^n F/I_q^{n+1} F$ of order $2^s$ in which every element has an $n$-fold Pfister representative, we associate an invariant in $I_q^{n+1} F$ which lives inside $I_q^{n+s-1} F$ when the forms in $S$ have a common $(n-1)$-fold Pfister factor. We study the properties of this invariant and compute it explicitly in a few interesting cases.
\end{abstract}

\end{frontmatter}

%\newpage

\section{Introduction}

The linkage of quaternion algebras has been the subject of several papers in recent years. Two quaternion algebras over a field $F$ are linked if they share a common quadratic field extension $K$ of $F$. When $\operatorname{char}(F)=2$, $K/F$ can be either separable or inseparable, and so we may specify the type of linkage accordingly.
Draxl proved in \cite{Draxl:1975} that if two quaternion algebras are inseparably linked then they are separably linked as well.
In \cite{Lam:2002} Lam simplified the proof and provided a counterexample to the converse. This result was generalized in \cite{Chapman:2015} to cyclic $p$-algebras of any prime degree. In \cite{ElduqueVilla:2005} a necessary and sufficient condition was given for two separably linked quaternion algebras to be inseparably linked.
In \cite{Faivre:thesis} Faivre generalized these results to quadratic $n$-fold Pfister forms for any $n \geq 2$.

The notion of linkage can be studied for more than two quaternion algebras:
we say that a set $S=\{Q_1,\dots,Q_t\}$ of quaternion algebras over $F$ is linked if there exists a quadratic field extension $K$ of $F$ such that $K \subseteq Q_i$ for all $i \in \{1,\dots,t\}$. We say that $S$ is tight if every class in the subgroup $\langle Q_1,\dots,Q_t \rangle$ of $\Brp(F)$ has a quaternion representative. Peyre gave in \cite{Peyre1995} an example of a tight set of three quaternion algebras over a field $F$ with $\operatorname{char}(F) \neq 2$ that is not linked. In \cite{Sivatski:2014} Sivatski defined an invariant in $I_q^3 F/I_q^4 F$ of tight sets of three quaternion algebras that vanishes when the set is linked. An example of a tight but not linked set of three quaternion algebras over a field of characteristic not 2 with a trivial Sivatski invariant was recently constructed in \cite{QueguinerTignol:2015}.

In this paper we study the linkage of sets of quadratic Pfister forms. Quaternion algebras can be identified with their norm forms which are quadratic $2$-fold Pfister forms. We say that a set of $n$-fold Pfister forms is {\bf{linked}} if there exists an $(n-1)$-fold Pfister form which is a common factor to all the forms in the set.
Generalizing the notion introduced above, we say that a set of $n$-fold Pfister forms is {\bf{tight}} if every class in the group it generates in $I_q^nF/I_q^{n+1}F$ is represented by an $n$-fold Pfister form. In \cite[Section~2]{ELW79} such sets are said to ``generate a linked group''. Every linked set is tight, and our goal is to propose further restrictions on a tight set, bringing it closer to linkage.

Let $S$ be a tight set of $n$-fold Pfister forms generating a group of order $2^s$ in $I_q^nF/I_q^{n+1}F$. We associate to $S$ an invariant $\Sigma_S \in I_q^{n+1}F$ (see \Dref{Sigmadef}) that satisfies $\Sigma_S \in I_q^{n+s-1}F$ when $S$ is linked.

Over fields of characteristic 2, we show that a right-linked pair of $n$-fold Pfister forms is left-linked if and only if the associated invariant is trivial (see \Tref{twoforms}).
The notion of linkage for larger sets is more complicated.
Just as a right-linked set is always tight, a left-linked set must be {\bf{strongly tight}}, namely it generates a subgroup whose elements are represented by Pfister forms not only in the quotient $I^n_qF/I^{n+1}_qF$, but also in $I^n_qF$ itself. When $S$ is strongly tight, the invariant $\Sigma_S$ is trivial.  In \Sref{sec:7} we show that a tight set $S$ is strongly tight if and only if the invariants associated to its subsets are all trivial. This allows us to conclude that a right-linked set which is pairwise left-linked is in fact strongly tight.

\Sref{quaternion} is devoted to quaternion algebras. We observe that the anisotropic part of the invariant associated to a set of $s$ $2$-fold Pfister forms has dimension at most $2^{s+1}$. It follows that if $S$ is right-linked and contains some pair of left-linked forms then $\Sigma_S = 0$. In Sections~\ref{stnolink} and \ref{higherintersect} we develop an additive notation for quadratic Pfister forms in characterstic $2$ \`{a} la Jacobson, and provide a set of $n$-fold forms of size $n+1$ which is strongly tight but not linked. This answers \cite[Open Question 1]{Sivatski:2014} in the negative when the characteristic is $2$.
We conclude the paper with an example of a tight triplet of quaternion algebras over a field $F$ of characteristic 2 whose invariant belongs to a nonzero class in $I^3_qF/I^4_qF$.

\section{Quadratic Pfister Forms}\label{sec:2}

We recall the basic structure theory of quadratic forms. For reference see \cite{EKM}. Let $F$ be a field. If $\operatorname{char}(F) \neq 2$ we tacitly assume that $F$ contains a square root of $-1$.

The hyperbolic plane $\HQ$ is the unique (up to isometry) two-dimensional nonsingular isotropic quadratic form $(u_1,u_2) \mapsto u_1 u_2$.

The diagonal bilinear form $b(u,v) = \alpha_1 u_1v_1+\dots+\alpha_n u_nv_n$ is denoted by $\Qf{\alpha_1,\dots,\alpha_n}$. The same notation stands for the quadratic form $b(u,u)$. The notation $[\alpha,\beta]$ stands for the 2-dimensional quadratic form $\alpha u_1^2+u_1 u_2+\beta u_2^2$.
Every quadratic form over $F$ can be written as an orthogonal sum
\begin{equation*}\label{genform}
\Qf{\alpha_1,\dots,\alpha_t} \perp [\beta_1,\gamma_1] \perp \cdots \perp [\beta_r,\gamma_r].
\end{equation*}
In characteristic not $2$ we may assume $r = 0$ (indeed $[\beta,\gamma] = \Qf{\beta,\gamma-(4\beta)^{-1}}$ and $[0,0] = \Qf{1,-1}$), and the form is nonsingular if $\alpha_1,\dots,\alpha_t \neq 0$. In characteristic~$2$, $r$ and $t$ are unique and the form is nonsingular if $t = 0$.

Every nonsingular quadratic form $\varphi$ decomposes uniquely as
$$\varphi \simeq i \times \HQ \perp \an{\varphi}$$
where $\an{\varphi}$ is anisotropic. % and $\Qf{0}$ is the 1-dimensional zero form.
We call $i$ the Witt index of $\varphi$ and denote it by $i_W(\varphi)$.
%Symmetric nonsingular bilinear forms, modulo hyperbolic forms, give rise to the Witt ring $WF$ of bilinear forms over $F$.

% A quadratic form is nondegenerate if it remains nonsingular under any field extension. In characteristic not $2$, every nonsingular form is nondegenerate. In characteristic~$2$, a form $i \times \HQ \perp \an{\varphi}$ is nondegenerate if and only if $t \leq 1$ in the  decomposition \eq{genform} of $\an{\varphi}$.

Two nonsingular quadratic forms are Witt equivalent if their underlying anisotropic subforms are isometric.
The group of Witt equivalence classes of even dimensional nonsingular quadratic forms over $F$, with respect to orthogonal sum, is denoted by $I_q^1 F$. This is a module over the Witt ring $WF$ (of symmetric nondegenerate bilinear forms, modulo metabolic forms), with respect to tensor product.

We abuse the terminology by often identifying a quadratic form with its Witt equivalence class.
The group $I_q^1 F$ is generated by the forms $[\alpha,\beta]$;  if $\operatorname{char}(F) \neq 2$, it is also generated by the forms $\Qf{\alpha,\beta}$ ($\alpha,\beta \in \mul{F}$). Since we assume $\sqrt{-1} \in F$ when the  characteristic is not $2$, $I_q^1F$ is annihilated by the element $\Qf{1,1}$ of the Witt ring.

The bilinear forms $\Pf{\beta} = \Qf{1,\beta}$ are called (bilinear) $1$-fold Pfister forms. These forms span the basic ideal $IF$ of $WF$. Powers of $IF$ are denoted $I^nF$. The tensor products $\Pf{\beta_1,\dots,\beta_n} = \Pf{\beta_1} \tensor \cdots \tensor \Pf{\beta_n}$ are called bilinear $n$-fold Pfister forms.

%%This is correct; but so is the short version.
%The 1-fold Pfister forms are the forms of the shape $\Qf{1,\alpha}$ denoted by $\Pf{\alpha}$ with $\alpha \in F^\times$ if $\operatorname{char}(F) \neq 2$, and of the shape $[1,\alpha]$ denoted by $\langle \langle \alpha]]$ if $\operatorname{char}(F)=2$.
The quadratic form $[1,\alpha]$ is called a (quadratic) $1$-fold Pfister form, and denoted by $\MPf{\alpha}$ (in characteristic not $2$ the customary notation is $\Pf{\alpha} = \Qf{1,\alpha}$, but $\Pf{\alpha} = \MPf{\alpha+1/4}$). %For a bilinear form $b$ and a quadratic form $q'(u) = b'(u,u)$, the tensor product $b \tensor q'$ is the quadratic form associated to $b \tensor b'$.
%The 1-fold Pfister forms are the forms of the shape $\langle 1,\alpha \rangle$ denoted by $\langle \langle \alpha \rangle \rangle$ with $\alpha \in F^\times$ if $\operatorname{char}(F) \neq 2$, and of the shape $[1,\alpha]$ denoted by $\langle \langle \alpha]]$ if $\operatorname{char}(F)=2$.
%Since $[\alpha,\beta]$ is isometric to $\alpha \cdot [1,\alpha \beta]$ and $\Qf{\alpha,\beta}$ is isometric to $\alpha \cdot\Pf{1,\alpha^{-1} \beta}$, the group $I_q F$ is generated by scalar multiples of $1$-fold Pfister forms.
%In particular,
For any quadratic form $\varphi$ and $\beta_1,\dots,\beta_n \in F^\times$, $\Qf{\beta_1,\dots,\beta_n} \otimes \varphi = \beta_1 \varphi \perp \dots \perp \beta_n \varphi$.
For any integer $n \geq 2$, we define the quadratic $n$-fold Pfister form $\MPf{ \beta_1,\dots,\beta_{n-1},\alpha}$ as $\Pf{\beta_{1},\dots,\beta_{n-1}} \otimes \MPf{\alpha}$.
We denote by $P_n(F)$ the set of $n$-fold Pfister forms.
A~Pfister form is isotropic if and only if it is hyperbolic.
We define $I_q^n F$ to be the subgroup of $I_q^1 F$ generated by the scalar multiples of quadratic $n$-fold Pfister forms. It follows that $I^mF \cdot I_q^kF = I^{m+k}_qF$ for every $m,k \geq 0$.

% , and $GP_n(F)$ the set of their scalar multiples. It is easy to see that $I^1_q F$ is generated by $GP_1(F)$. We define $I_q^n F$ to be the subgroup of $I_q^1 F$ generated by $GP_n(F)$.

\section{Linkage and tightness}\label{sec:3}

If $b$ is a bilinear $m$-fold Pfister form and $\varphi$ is a quadratic $k$-fold Pfister form, then $b \tensor \varphi$ is a quadratic $(m+k)$-fold Pfister form. In this case we say that $b$ is a left divisor and $\varphi$ is a right divisor of $b \tensor \varphi$. A set $S$ of quadratic $(m+k)$-fold Pfister forms is said to be $k$-right-linked if there is a quadratic $k$-fold Pfister form $\varphi$ which is a right divisor of every element of $S$.
% = there exist $\alpha_1,\dots,\alpha_n \in F$ such that every form in $S$ is of the shape $\langle \langle \beta_1,\dots,\beta_{m-n},\alpha_1,\dots,\alpha_n ]]$ for some $\beta_1,\dots,\beta_{m-n} \in F$.
Similarly, $S$ is said to be $m$-left-linked if there is a bilinear $m$-fold Pfister form $b$ which is a left divisor of every element of $S$.
% = exist $\alpha_1,\dots,\alpha_n \in F$ such that every form in $S$ is of the shape $\langle \langle \alpha_1,\dots,\alpha_n,\beta_1,\dots,\beta_{m-n} ]]$ for some $\beta_1,\dots,\beta_{m-n} \in F$.

Throughout the paper, we are only interested in linkage where the free part is a (bilinear or quadratic) $1$-fold Pfister form:
\begin{defn}\label{lrlink}
A set of quadratic $n$-fold Pfister forms is {\bf{right-linked}} if the forms are $(n-1)$-right-linked and {\bf{left-linked}} if they are $(n-1)$-left-linked.
\end{defn}

If $\operatorname{char}(F) \neq 2$ there is a one-to-one correspondence between quadratic and symmetric bilinear forms, so the notions of left- and right- linkage coincide.

The following fact allows us to introduce a necessary condition for a set of forms to be linked:
\begin{rem}[{\cite[Cor.~23.9]{EKM}}]\label{unique}
If two $n$-fold Pfister forms are equivalent modulo $I_q^{n+1} F$ then they are isometric.
% In fact what 23.9 says is that if $q,q'$ are both scalar multiples of n-fold Pfister forms, and equivalent mod I^{n+1}, then they are similar; and if furthermore $D(q)\cap D(q')$ is nonempty, then they are isometric.
% But all Pfister forms share the value 1 in common, so our version follows.
\end{rem}

\begin{defn}
We say that a set of quadratic $n$-fold Pfister forms is {\bf{tight}} if every element of the group it generates in $I_q^{n}F/I_q^{n+1}F$ is represented by a Pfister form (which is unique, by \Rref{unique}). The set is {\bf{strongly tight}} if the group it generates in $I_q^nF$ consists of Pfister forms.
\end{defn}
%A {\bf{proto-group of Pfister forms}} is a set $S$ of distinct $(n+1)$-Pfister forms whose images in $I_q^{n+1} F/I_q^{n+2} F$ form a subgroup (necessarily of order $\card{S}$).
Obviously, a strongly tight set is tight.
A subset of a (strongly) tight set is itself (strongly) tight. In the case of quaternions (namely $n = 2$), a group generated by a tight set is called in \cite{QueguinerTignol:2015} a `quaternionic subgroup' (we could call a group generated by a tight set a `symbolic subgroup').
%(we decided to avoid `formic subgroup').
%
% To illustrate these notions, let us consider a set of $2$-fold Pfister forms $\varphi_1,\varphi_2,\varphi_3$. If the set is arbitrary, then $\varphi_i+\varphi_j$ is not a Pfister form even modulo $I_q^3$.

By \cite[Proposition 24.5]{EKM}, every two forms in a tight set are right-linked.
We move on to show that every right-linked set is tight, and every left-linked set is strongly tight.

%As we shall see below, not every tight set of Pfister forms is right-linked. Our leading question is therefore: How can one measure the failure of linkage in a tight set?

%\begin{rem}
%%Modulo $I^2(F)$, $\Pf{b_1} + \Pf{b_2} \equiv \Pf{b_1b_2}$. Therefore, for every $\varphi \in P_n(F)$, $\Pf{b_1}\tensor\varphi + %\Pf{b_2}\tensor \varphi \equiv \Pf{b_1b_2} \tensor \varphi \pmod{I^{n+2}_q(F)}$. In fact,
%Since $\Pf{b} \perp \Pf{b'} \sim \Pf{bb'}+\Pf{b,b'} \equiv \Pf{bb'} \pmod{I_q^2F}$,
%$\Pf{b} \tensor \varphi \perp \Pf{b'} \tensor \varphi \equiv \Pf{bb'} \tensor \varphi \pmod{I_q^{n+1}F}$ for every $(n-1)$-fold Pfister %form $\varphi$.
%
%Let $\varphi$ be an $(n-1)$-fold Pfister form. By \cite[Corollary 23.10]{EKM}, $\Pf{b_1b_2} \tensor \varphi$ is the unique $n$-fold Pfister %form equivalent to $\Pf{b_1}\tensor\varphi + \Pf{b_2}\tensor \varphi$ modulo $I_q^{n+1} F$.
%\end{rem}

\begin{prop}\label{3.2}
Let $S$ be a right-linked set of $n$-fold Pfister forms. Then there are a quadratic $(n-1)$-fold Pfister form $\varphi$ and elements $\beta_1,\dots,\beta_s \in \mul{F}$ such that the elements of the group generated by $S$ in $I_q^nF/I_q^{n+1}F$ are represented by the Pfister forms $\Pf{\beta_{i_1}\cdots \beta_{i_k}} \tensor \varphi$ for the subsets $\set{i_1,\dots,i_k} \sub \set{1,\dots,s}$.
\end{prop}
\begin{proof}
By the assumption, the elements of $S$ are of the form $\Pf{\beta_i} \tensor \varphi$ where $\varphi$ is a fixed $(n-1)$-fold Pfister form. Since $\Pf{\beta} \perp \Pf{\beta'} \sim \Pf{\beta\beta'} \perp \Pf{\beta,\beta'} \equiv \Pf{\beta\beta'} \pmod{I^2F}$, the forms $\Pf{\beta_{i_1}\cdots \beta_{i_k}} \tensor \varphi$ represent the elements in the generated subgroup.
\end{proof}

In particular,
\begin{cor}\label{rl->t}
A right-linked set is tight.
\end{cor}

\subsection{Left-linkage in characteristic~$2$}

Now consider left-linkage and assume $F$ has characteristic~$2$.
\begin{prop}\label{3.5}
In characteristic~$2$, a left-linked set of quadratic $n$-fold Pfister forms is strongly tight.

Moreover if $S$ is left-linked then there are a bilinear $(n-1)$-fold Pfister form $b$ and elements $\alpha_1,\dots,\alpha_s \in F$ such that the elements of the group generated by $S$ in $I_q^nF$ are the Pfister forms $b \tensor \MPf{\alpha_{i_1} + \cdots +\alpha_{i_k}}$ for the subsets $\set{i_1,\dots,i_k} \sub \set{1,\dots,s}$.
\end{prop}
\begin{proof}
Same as the proof of \Pref{3.2}, using the fact that in characteristic~$2$, $\MPf{\alpha}+\MPf{\alpha'} \sim \MPf{\alpha+\alpha'}$ is an equality in $I_q^1 F$, and not just modulo $I_q^2F$ \cite[Example~7.23]{EKM}.
% An explicit proof: pass from $x,y,x',y'$ to $x,y+y',x+x',y'$.
\end{proof}

We prove a somewhat stronger criterion for strong tightness in \Pref{cpll}.

\section{An obstruction to strong tightness}\label{sec:4}

Given a tight set $S$ of $n$-fold Pfister forms, we now define an invariant $\Sigma_S$ which vanishes if the set is strongly tight (and thus if the set is left-linked in characteristic~$2$), and is a higher Pfister form if the set is right-linked. The invariant defined by Sivatski for triplets of quaternion algebras is the class of $\Sigma_S$ modulo $I_q^4 F$ in the special case of $\operatorname{char}(F) \neq 2$, $n=2$ and $|S|=3$ (more details are given in \Sref{quaternion} below).

\begin{defn}\label{Sigmadef}
Let $S$ be a finite tight set of (more than one) $n$-fold Pfister forms. We define $\Sigma_S$ to be the sum in~$I_q^nF$ of the unique Pfister representatives of the elements of the subgroup generated by the image of~$S$ in~$I_q^nF/I_q^{n+1}F$ (which exist by assumption and are unique by \Rref{unique}, so $\Sigma_S$ is well defined).
\end{defn}
%If $S$ consists of a single form, it would be equal to $\Sigma_S$, which we do not care to define in this case.

\begin{prop}\label{down}
\begin{enumerate}
\item For any tight set $S$ of $n$-fold Pfister forms, $\Sigma_S \in I_q^{n+1}F$.
\item If $S$ is strongly tight then $\Sigma_S = 0$.
\end{enumerate}
\end{prop}
\begin{proof}
The image of $\Sigma_S$ in $I_q^{n} F/I_q^{n+1} F$ is the sum of all the elements in a noncyclic group of exponent $2$, which is zero; this proves the first claim. The second claim follows from the same argument, applied to $I_q^nF$ rather than $I_q^nF/I_q^{n+1}F$
\end{proof}

\begin{prop}\label{sep}
Let $S$ be a right-linked set of $n$-fold Pfister forms, generating a group of order $2^s$ modulo $I_q^{n+1}F$. Then $\Sigma_S$ is the Witt class of an $(n+s-1)$-fold Pfister form.
\end{prop}
\begin{proof}
By \Rref{3.2}, and using the notation there, the elements of the group generated by $S$ are represented by the forms $\Pf{\beta_{i_1}\cdots \beta_{i_k}} \tensor \varphi$, where $\set{i_1,\dots,i_k} \sub \set{1,\dots,s}$.

One proves by induction on $t = 1,\dots,s$ that for every $\alpha,\beta_1,\dots,\beta_t \in \mul{F}$, $$\sum_{\set{i_1,\dots,i_k} \sub \set{1,\dots,t}} \Pf{\alpha\beta_{i_1}\cdots \beta_{i_k}} = \Pf{\beta_1,\dots,\beta_t} + \Pf{\alpha,\beta_1,\dots,\beta_t}$$
\sloppy in $I_q^1 F$. Taking $\alpha = 1$, the sum of the left $1$-fold factors of the representatives in the group generated by $S$ is seen to be the $s$-fold form $\Pf{\beta_1,\dots,\beta_s}$, so $\Sigma_S = \Pf{\beta_1,\dots,\beta_s} \tensor \varphi$ (equality in the Witt group).
\end{proof}

% In case $s=2$, the converse of this Proposition also holds by \cite[Proposition 24.5]{EKM}. See Section \ref{quaternion} for a discussion on counterexamples to the converse of this statement for three quaternion algebras (namely $s = 3$ and $n = 2$).

\forget
\begin{lem}
Assuming $\operatorname{char}(F)=2$, the direct sum of the forms $\MPf{\alpha_1,\dots,\alpha_n,\beta_1}$ and $\MPf{\alpha_1,\dots,\alpha_n,\beta_2}$ is Witt equivalent to $\MPf{\alpha_1,\dots,\alpha_n,\beta_1+\beta_2}$.
\end{lem}
\begin{proof}
Short proof: $[1,\beta_1]+[1,\beta_2] \sim [1,\beta_1+\beta_2]$.

One can write the form $\MPf{\alpha_1,\dots,\alpha_n,\beta_1}$ as the direct sum of $\alpha \cdot [1,\beta_1]$ where $\alpha$ ranges over the partial products of $a_1,\dots,a_n$. Similarly $\MPf{\alpha_1,\dots,\alpha_n,\beta_2}$ is the direct sum of the terms $\alpha \cdot [1,\beta_2]$ and $\MPf{\alpha_1,\dots,\alpha_n,\beta_1+\beta_2}$ is the direct sum of the terms $\alpha \cdot [1,\beta_1+\beta_2]$.
It is therefore enough to prove that $\alpha \cdot [1,\beta_1] \perp \alpha \cdot [1,\beta_2]$ is Witt equivalent to $\alpha \cdot [1,\beta_1+\beta_2]$.
Now, $\alpha \cdot [1,\beta_1] \simeq [\alpha^{-1},\alpha \beta_1]$ and $\alpha \cdot [1,\beta_2] \simeq [\alpha^{-1},\alpha \beta_2]$, and so $\alpha \cdot [1,\beta_1] \perp \alpha \cdot [1,\beta_1]$ is isometric to $[\alpha^{-1},\alpha \beta_1] \perp [\alpha^{-1},\alpha \beta_2]$, which in turn is Witt equivalent to $[\alpha^{-1},\alpha (\beta_1+\beta_2)]$, and this is isometric to $\alpha \cdot [1,\beta_1+\beta_2]$.
\end{proof}
\forgotten

\section{Left- and Right- Linkage of pairs}\label{LeftRight}

Throughout this section we assume $\operatorname{char}(F)=2$, and fix $n \geq 2$. We start by recalling some results from \cite{Faivre:thesis}:

\begin{cor}[{\cite[Corollary 2.1.4]{Faivre:thesis}}]
If two Pfister forms are left-linked then they are also right-linked.
\end{cor}

However, there are pairs of right-linked forms which are not left-linked, \cite[Proposition 2.5.1]{Faivre:thesis}.
%\begin{prop}[{\cite[Proposition 2.5.1]{Faivre:thesis}}]
%In characteristic~$2$, there are pairs of right-linked forms which are not left-linked.
%Assume $I_q^{n} F \neq 0$. Then for any $n$-fold Pfister form $\phi$ over $F$ there exists an $n$-fold Pfister form $\psi$ over the field of rational functions $K=F(x_1)$  such that $\phi_K$ and $\psi$ are right-linked but not left-linked.
%\end{prop}

If $b$ is a bilinear $n$-fold Pfister form, the pure subform is the unique form $b'$ such that $b = \Qf{1} \perp b'$. For a quadratic Pfister form $\varphi$, the pure subform is $\varphi' = \psi \perp \Qf{1}$, where $\psi$ is the unique nonsingular quadratic form such that $\varphi \perp \Qf{1} \isom \HQ \perp \psi \perp \Qf{1}$. It follows that if $\varphi = b \tensor \MPf{\beta,\alpha}$, then
$$\varphi' = b' \tensor \MPf{\alpha} \perp b \tensor \Qf{\beta} \tensor \MPf{\alpha} \perp  \Qf{1}.$$
% And the pure subform of $<<\alpha]]$ is $<1>$.

\begin{thm}[{\cite[Theorem 2.5.5]{Faivre:thesis}}]\label{5.2}
Two $n$-fold Pfister forms $\phi$ and $\psi$ are left-linked if and only if $i_W(\phi' \perp \psi') \geq 2^{n-1}-1$.
\end{thm}

Using this criterion of Faivre for left-linkage of $n$-fold Pfister forms, we can give a criterion for left-linkage of right-linked forms:
\begin{thm}\label{twoforms}
Two right-linked $n$-fold Pfister forms
$$\varphi = \MPf{\beta,\alpha_1,\dots,\alpha_{n-1}}, \qquad \psi=\MPf{\gamma,\alpha_1,\dots,\alpha_{n-1}}$$
are left-linked if and only if the $(n+1)$-fold Pfister form
$\pi = \MPf{\gamma,\beta,\alpha_1,\dots,\alpha_{n-1}}$ is hyperbolic.
\end{thm}
\begin{proof}
Write $b = \Pf{\alpha_1,\dots,\alpha_{n-2}}$, so that $\varphi = b \tensor \MPf{\beta,\alpha_{n-1}}$ and $\psi = b \tensor \MPf{\gamma,\alpha_{n-1}}$. By definition $$\varphi' = b' \tensor \MPf{\alpha_{n-1}} \perp b \tensor \Qf{\beta} \tensor \MPf{\alpha_{n-1}} \perp  \Qf{1}$$
and $$\psi' = b' \tensor \MPf{\alpha_{n-1}} \perp b \tensor \Qf{\gamma} \tensor \MPf{\alpha_{n-1}} \perp  \Qf{1};$$
since %$\Pf{\beta}\perp \Pf{\gamma} = \HQ \perp \Qf{\beta,\gamma}$ and
$\Qf{1} \perp \Qf{1} \isom \Qf{1} \perp \Qf{0}$, we have that $$\varphi' \perp \psi' = (2^{n-1}-2) \times \HQ \perp \Qf{\beta,\gamma} \tensor b  \tensor \MPf{\alpha_{n-1}} \perp  \Qf{1,0},$$
and so the Witt index $i_W(\varphi' \perp \psi')$ is equal to $2^{n-1}-2$ plus the Witt index of the form $\theta = \Qf{1} \perp \Qf{\beta,\gamma} \tensor b  \tensor \MPf{\alpha_{n-1}}$. By Faivre's criterion, $\varphi$ and $\psi$ are left-linked if and only if $\HQ \sub \theta$.

But since $\theta \sub \Qf{1,\beta,\gamma,\beta \gamma} \tensor b  \tensor \MPf{\alpha_{n-1}} \isom \pi$, $\theta$ is a Pfister neighbor of $\pi$, and thus $\theta$ is isotropic if and only if $\pi$ is hyperbolic.%they are isotropic together.

% Since $\Qf{1} \sub b \tensor \MPf{\alpha_{n-1}}$, we have that $\theta = \Qf{1} \perp \Qf{\beta,\gamma} \tensor b \tensor \MPf{\alpha_{n-1}} \sub \Qf{1,\beta,\gamma,\beta\gamma} \tensor b \tensor \MPf{\alpha_{n-1}}$, so if $\HQ \sub \theta$, the Pfister form $\Pf{\beta,\gamma} \tensor b \tensor \MPf{\alpha_{n-1}}$ is isotropic and thus hyperbolic.
% On the other hand if $\Pf{\beta,\gamma} \tensor b \tensor \MPf{\alpha_{n-1}}$ is hyperbolic, then $\Qf{\beta,\gamma} \otimes b \tensor \MPf{\alpha_{n-1}} \simeq \Qf{1,\beta\gamma}\otimes b \tensor \MPf{\alpha_{n-1}}$.
% Adding $\Qf{1}$ to both forms, we get $\theta \simeq \Qf{1,\beta\gamma} \tensor b \tensor \MPf{\alpha_{n-1}} \perp \Qf{1}$ which is obviously isotropic.
\end{proof}

\forget %% A trivial special case of the above
\begin{thm}\label{twoforms}
Two right-linked $2$-fold Pfister forms
$$\phi = \MPf{\beta,\alpha_1}, \qquad \psi=\MPf{\gamma,\alpha_1}$$
are left-linked if and only if the $3$-fold Pfister form
$\MPf{\gamma,\beta,\alpha_1}$ is hyperbolic.
\end{thm}
\begin{proof}
Write $b = \Qf{1}$, so that $\varphi = b \tensor \MPf{\beta,\alpha_{1}}$ and $\psi = b \tensor \MPf{\gamma,\alpha_{1}}$. By definition $$\varphi' =  \Qf{\beta} \tensor \MPf{\alpha_{n-1}} \perp  \Qf{1}$$
and $$\psi' =  \Qf{\gamma} \tensor \MPf{\alpha_{1}} \perp  \Qf{1};$$
since %$\Pf{\beta}\perp \Pf{\gamma} = \HQ \perp \Qf{\beta,\gamma}$ and
$\Qf{1} \perp \Qf{1} \isom \Qf{1} \perp \Qf{0}$, we have that $$\varphi' \perp \psi' = \Qf{\beta,\gamma} \tensor \MPf{\alpha_{1}} \perp  \Qf{1,0},$$
and so the Witt index $i_W(\varphi' \perp \psi')$ is equal to the Witt index of the form $\theta = \Qf{\beta,\gamma} \tensor \MPf{\alpha_{1}} \perp  \Qf{1}$. By Faivre's criterion, $\varphi$ and $\psi$ are left-linked if and only if $\HQ \sub \theta$.

Since $\Qf{1} \sub \MPf{\alpha_{1}}$, we have that $\theta = \Qf{1} \perp \Qf{\beta,\gamma} \tensor \MPf{\alpha_{1}} \sub \Qf{1,\beta,\gamma,\beta\gamma} \tensor \MPf{\alpha_{1}}$, so if $\HQ \sub \theta$, the Pfister form $\Pf{\beta,\gamma} \tensor  \MPf{\alpha_{1}}$ is isotropic and thus hyperbolic.

On the other hand if $\Pf{\beta,\gamma} \tensor \MPf{\alpha_{1}}$ is hyperbolic, then $\Qf{\beta,\gamma} \otimes \MPf{\alpha_{1}} \simeq \Qf{1,\beta\gamma}\otimes \MPf{\alpha_{1}}$.
Adding $\Qf{1}$ to both forms, we get $\theta \simeq
\Qf{1,\beta\gamma} \tensor \MPf{\alpha_{1}} \perp \Qf{1}$ which is obviously isotropic.
\end{proof}
\forgotten

%\begin{exmpl}\label{twoforms:n=2}
%The $2$-fold Pfister forms $\MPf{\beta,\alpha}$ and $\MPf{\beta',\alpha}$
%are left-linked if and only if the $3$-fold Pfister form
%$\MPf{\beta,\beta',\alpha}$ is hyperbolic.
%\end{exmpl}

\begin{cor}
Suppose $I_q^{n+1} F=0$. Then two $n$-fold Pfister forms over $F$ are right-linked if and only if they are left-linked.
\end{cor}

%In Section \ref{LeftRight} we shall see that when $\operatorname{char}(F)=s=2$, $\Sigma_S$ being hyperbolic is a sufficient condition for $S$ to be $n$-left-linked.
Fix $s \geq 3$. We give an example of a right-linked set $S$ (of size $s$), which is not pairwise left-linked, despite the fact that $\Sigma_S = 0$:
\begin{exmpl}\label{counter3.6}{\rm
Let $F=k(\beta_1,\dots,\beta_s,\alpha_1,\dots,\alpha_{n-1})$ be the function field in $n+s-1$ algebraically independent variables over a field $k$ of characteristic~$2$.
Let $S$ be the set of the $n$-fold Pfister forms $\MPf{\beta_i,\alpha_1,\dots,\alpha_{n-1}}$ where $i = 1,\dots,s$; it is right-linked, and therefore tight. By \Pref{sep}, $\Sigma_S$ is the Witt class of the generic $(n+s)$-fold Pfister form $\sigma =\MPf{\beta_1,\dots,\beta_s,\alpha_1,\dots,\alpha_n}$.

\sloppy Let $K$ be the function field of the form $\sigma$.
Consider any two Pfister forms $\MPf{\beta,\alpha_1,\dots,\alpha_{n-1}}$ and $\MPf{\beta', \alpha_1,\dots,\alpha_{n-1}}$, where $\beta$ and $\beta'$ are distinct products of subsets of the $\beta_i$.
The form $\MPf{\beta,\beta',\alpha_1,\dots,\alpha_{n-1}}$ is not split by $K$
% because anisotropic $(n+1)$-fold Pfister forms are not split by the function fields of $(n+s-1)$-fold Pfister forms with $s \geq 3$
(see \cite[Theorem 1.1]{HoffmannLaghribi2006}), so by \Tref{twoforms}, the two forms are not left-linked over $K$.

Therefore, by restricting everything to $K$, we obtain a right-linked set of forms with $\Sigma = 0$, where no two forms are left-linked.}
\end{exmpl}

\section{Tightness and strong tightness of pairs}\label{sec:6}

In this section we assume $S = \set{\varphi_1,\varphi_2}$ is a set of two quadratic $n$-fold Pfister forms, and compare the notions of left- and right- linkage with tightness and strong tightness.

\begin{prop}\label{2LL}
The pair $\varphi_1,\varphi_2$ is right-linked if and only if it is tight.
\end{prop}
\begin{proof}
This is \Cref{rl->t} and \cite[Prop.~24.5]{EKM} which was quoted earlier.
\end{proof}

\begin{prop}\label{2RL}
Let $S = \set{\varphi_1,\varphi_2}$ and assume $F$ has characteristic~$2$. The following are equivalent:
\begin{enumerate}
\item $S$ is left-linked.
\item $S$ is strongly tight.
\item $S$ is tight and $\Sigma_S = 0$.
\end{enumerate}
\end{prop}
\begin{proof}
A left-linked set is strongly tight by \Pref{3.5}. A strongly tight set $S$ is tight and has $\Sigma_S = 0$ by \Pref{down}.2. It remains to show that if $S = \set{\varphi_1,\varphi_2}$ is a tight pair and $\Sigma_S = 0$ then $S$ is left-linked. By \Pref{2LL}, $S$ is right-linked, so we can write them as $\varphi_i = \Pf{\alpha_i} \tensor \psi$ for a quadratic Pfister form $\psi$. The invariant $\Sigma_S$ was computed in \Pref{sep} to be the Witt class of $\Pf{\alpha_1,\alpha_2} \tensor \psi$. By \Tref{twoforms}, $S$ is left-linked.
\end{proof}

This criterion can be generalized:
\begin{prop}\label{cpll}
When $\mychar F = 2$, the following are equivalent for a tight set $S$ of $n$-fold Pfister forms.
\begin{enumerate}
\item Every two Pfister forms representing elements of the group generated by $S$ in $I_q^nF/I_q^{n+1}F$ are left-linked;
\item $S$ is strongly tight.
\end{enumerate}
\end{prop}
\begin{proof}
Let $G$ denote the group generated by $S$ in $I_q^nF/I_q^{n+1}F$. Since $S$ is tight, for every $g \in G$ there is a (unique) quadratic $n$-fold Pfister form $\varphi_g \in g + I_q^{n+1}F$.

Assume $S$ is strongly tight. Let $g,g' \in G$. Clearly $\set{\varphi_g,\varphi_{g'}}$ is strongly tight, so by \Pref{2RL} they are left-linked. On the other hand, assume every $\varphi_g,\varphi_{g'}$ are left-linked. Let $g,g' \in G$ be distinct elements; then by \Pref{3.5} the pair $\varphi_g,\varphi_{g'}$ is strongly tight, and so $\varphi_{g+g'} = \varphi_g + \varphi_{g'}$ in $I_q^nF$. Since $\varphi_g$ has order $2$ in $I_q^nF$, we proved that $\set{\varphi_g \suchthat g \in G}$ is a subgroup of $I_q^nF$.
\end{proof}

\section{From tightness to strong tightness through invariants}\label{sec:7}

The notions of tightness and strong tightness, defined above for sets of quadratic Pfister forms, can be easily treated in  a general setting. The application to Pfister forms is transparent.

Let $V$ be a vector space over the field of two elements, $U \sub V$ a subspace, and $P \sub V$ a designated spanning set of $V$, with $0 \in P$, such that every coset of $V/U$ contains at most one element of $P$ (in our context $P$ is the set of quadratic $n$-fold Pfister forms, $V = I_q^nF$ and $U = I_q^{n+1}F$). A set $S \sub V$ is {\bf{tight}} if the generated subspace $\sg{S} \sub V$ is contained in $P+U$; and {\bf{strongly tight}} if $\sg{S} \sub P$. For a finite tight set $S$, we define $\Sigma_S$ to be the sum $\sum p_{S'} \in V$, where $p_{S'} \in P$ are the unique elements such that $\set{p_{S'}+U \suchthat S' \sub S}$ is the group generated by $S$ in $V/U$ (in particular $p_{\emptyset} = 0$). As long as $\card{\sg{S}}  > 2$, we have that $\Sigma_S \in U$.

Let us say that a set is {\bf{almost strongly tight}} if it is tight and every proper subset is strongly tight. A strongly tight set is almost strongly tight.
\begin{prop}\label{prep}
%% A more general version with the same content:
%Let $S = \set{\varphi_1,\dots,\varphi_s} \sub P$ be a tight set. Let $2 \leq k \leq s$. Assume every subset $S' \sub S$, of size $k-1$, is strongly tight. Then a subset $S' \sub S$ of size $k$ is strongly tight, if and only if $\Sigma_{S'} = 0$.
Let $S = \set{\varphi_1,\dots,\varphi_s} \sub P$ be an almost strongly tight set of finite cardinality $s > 1$. Then $S$ is strongly tight if and only if $\Sigma_{S} = 0$.
\end{prop}
\begin{proof}
We may assume $S$ is linearly independent, because otherwise the claim holds by assumption. For every $S' \sub S$, let $\Sigma(S') = \sum_{\varphi \in S'} \varphi$, and let $p_{S'}$ be the unique representative from $P$ of the element $\sum(S') + U$. If $S' \subsetneq S$ then $p_{S'} = \sum(S')$ because $S'$ is strongly tight.
% Moreover there is $p_S \in P$ such that $(\sum_{i \in S} \varphi_i) - p_S \in U$ because $S$ is tight, and
By definition
\begin{eqnarray*}
\Sigma_S & = & \sum_{S' \sub S} p_{S'}
\ =\   (\sum_{S' \subsetneq S} p_{S'}) + p_S
\  = \  (\sum_{S' \subsetneq S} \Sigma(S')) + p_S \\
& = & (\sum_{S' \sub S} \Sigma(S')) + p_S - \Sigma(S)
\  = \ p_S - \Sigma(S)
\end{eqnarray*}
because the sum of elements of a noncyclic group of exponent $2$ is zero.
By definition $S$ is strongly tight if and only if $\sum(S') \in P$ for every $S' \sub S$, and since $S$ is almost strongly tight by assumption, this is equivalent to $\Sigma(S) \in P$. But $\Sigma(S) = p_S -  \Sigma_S$  is in $P$ if and only if $\Sigma_S = 0$, because $\Sigma(S)+U$ intersects $P$ in at most one point.
\end{proof}

\begin{thm}\label{ladder}
Let $s > 1$. A tight set $S = \set{\varphi_1,\dots,\varphi_s} \sub P$ is strongly tight if and only if $\Sigma_{S'} = 0$ for every subset $S' \sub S$ of cardinality $> 1$.
\end{thm}
\begin{proof}
If $S$ is strongly tight then for every subset $S' \sub S$ we have that $\sg{S'} \sub P$, and the sum over a non-cyclic group of exponent $2$ is zero.

For the other direction, we prove that every subset of size $1 \leq k \leq \card{S}$ is strongly tight, by induction on $k$. Sets of size $1$ are strongly tight by definition. Let $S' \sub S$ be a subset of size $k > 1$. By induction we may assume every subset of size $k-1$ of $S'$ is strongly tight, namely $S'$ is almost strongly tight. By \Pref{prep}, $S'$ is strongly tight because $\Sigma_{S'} = 0$ by assumption.
\end{proof}

\subsection{Applications to quadratic forms}

Consider a set $S = \set{\varphi_1,\dots,\varphi_s}$ of $n$-fold Pfister forms. For $S$ to be left-linked, it is necessary that $S$ is strongly tight; and for that, it is necessary that $S$ is almost strongly tight.
\begin{cor}
Let $S = \set{\varphi_1,\dots,\varphi_s}$ be an almost strongly tight set of $n$-fold Pfister forms. Then $S$ is strongly tight if and only if $\Sigma_S = 0$. %\in I_q^{n+1}F$ is the only obstruction for
(This is \Pref{prep}).
\end{cor}

Assume $\operatorname{char}(F)=2$. For a set $S$ to be left-linked, it must be pairwise left-linked, and it must be strongly tight. It turns out that if $S$ is right-linked, then the two conditions are equivalent:
\begin{thm}\label{maincor++}
Let $S$ be a right-linked set of $n$-fold Pfister forms. The following are equivalent when $F$ has characteristic~$2$:
\begin{enumerate}
\item $S$ is pairwise left-linked.
\item $S$ is strongly tight. % In particular $\Sigma_S = 0$.
\end{enumerate}
\end{thm}
\begin{proof}
$(2) \implies (1)$: by \Pref{cpll}. %if $S$ is strongly tight, then every pair of elements is left-linked by \Pref{2RL}.
$(1) \implies (2)$: By assumption the elements of $S$ are $\Pf{\beta_i} \tensor \varphi$, $i = 1,\dots,s$, where $\beta_1,\dots,\beta_s \in \mul{F}$ and $\varphi$ is a quadratic $(n-1)$-fold Pfister. By \Tref{twoforms}, the forms $\Pf{\beta_i,\beta_j} \tensor \varphi$ are hyperbolic for every $i \neq j$. Moreover, $S$ is tight. By \Pref{3.2}, if $S' \sub S$ is a subset of size $>1$, then $\Sigma_{S'} = \Pf{\beta_{i_1},\dots,\beta_{i_k}} \tensor \varphi$ for a suitable set $\set{i_1,\dots,i_k} \sub \set{1,\dots,s}$. Every invariant of this form is divisible by one of the $\Pf{\beta_i,\beta_j} \tensor \varphi$, and is thus hyperbolic. By \Tref{ladder}, $S$ is strongly tight.
\end{proof}
We do not know if a right-linked set which is pairwise left-linked, is left-linked.

\begin{cor}\label{7.5}
Let $\alpha_1,\dots,\alpha_s \in \mul{F}$ and let $\psi$ be a quadratic $(n-1)$-fold Pfister form. Assume that the forms $\Pf{\alpha_i}\tensor \psi$ are pairwise left-linked. Then
$$\Pf{\alpha_1} \perp \cdots \perp \Pf{\alpha_s} \perp \Pf{\alpha_1\cdots\alpha_s}$$
annihilates $\psi$ in $I_q^nF$.
\end{cor}
\begin{proof}
The set of $n$-fold Pfister forms $\Pf{\alpha_i} \tensor \psi$ is right-linked and pairwise left-linked, by assumption. By \Tref{maincor++}, this set is strongly tight and therefore the sum of the forms is the unique Pfister form representing the sum modulo $I_q^{n+1}(F)$, which is $\Pf{\alpha_1\cdots\alpha_s} \tensor \psi$.
\end{proof}

\section{Quaternion Algebras}\label{quaternion}

Two quaternion algebras are linked if they share a common maximal subfield. The norm form of a quaternion algebra
\begin{equation}\label{standard}
\QA{\alpha}{\beta} = F \ideal{x,y \subjectto x^2+x=\alpha,\, y^2=\beta,\, y xy^{-1}=x+1}
\end{equation}
is $\MPf{\beta,\alpha}$, so left linkage translates to a common inseparable maximal subfield (of the form $F[y]$), and right linkage to a common separable maximal subfield (of the form $F[x]$). The results of \citep{Draxl:1975}, \cite{Lam:2002} and \cite{ElduqueVilla:2005} on the linkage of pairs of quaternion algebras were extended in \cite{Faivre:thesis} to $n$-fold Pfister forms.

Sivatski studies \cite{Sivatski:2014} triplets of quaternion algebras over a field $F$ of characteristic not $2$, containing $\sqrt{-1}$. If $S = \set{A,B,C}$ is (in our terminology) tight, he defines an invariant $\Sigma_S' \in I_q^3 F/I_q^4 F$ (which is the image of our invariant $\Sigma_S \in I_q^3F$). As pointed out by Sivatski, $\Sigma_S' = 0$ is a necessary condition for $A,B,C$ to be linked. Recall that in characteristic not $2$, the symbol notation for quaternion algebras is
$$(\alpha,\beta)_{2,F} = F\ideal{ x,y \subjectto x^2 = \alpha,\, y^2 = \beta,\, yxy^{-1} = -x}.$$
Sivatski shows in \cite[Cor.~3]{Sivatski:2014} that the algebras $(a,b)$, $(b,c)$ and $(c,a)$ over the function field $k(a,b,c)$ are tight (and thus pair-wise linked), but since in this case $\Sigma_S' = \Pf{a,b,c}$, the algebras cannot be linked as a triplet. He then asks for an example with $\Sigma_S' = 0$ which is not linked. Such an example was given in~\cite{QueguinerTignol:2015}.

\subsection{Pair of quaternion algebras}

Let $Q_1,Q_2$ be quaternion algebras over $F$, of characteristic~$2$, and let $\varphi_1,\varphi_2$ be the respective norm forms, which are $2$-fold Pfister forms. Write $Q_i = \QA{\alpha_i}{\beta_i}$, so that $\varphi_i = \MPf{\beta_i,\alpha_i}$. We provide explicit conditions for the existence of left- or right-linkage.
\begin{prop}\label{6.3}
The following conditions are equivalent.
\begin{enumerate}
\item $Q_1,Q_2$ have a common inseparable maximal subfield.
\item $\varphi_1,\varphi_2$ are left-linked.
\item The form $\Alb' = \beta_1[1,\alpha_1] \perp \beta_2[1,\alpha_2] \perp \Qf{1}$ is isotropic.
\end{enumerate}
\end{prop}
\begin{proof}
The equivalence of the first two conditions is trivial. Taking $n = 2$ in  \Tref{5.2}, since $2^{n-1}-1 = 1$,
the $2$-fold Pfister forms $\varphi_1$ and $\varphi_2$ are left-linked if and only if $\varphi_1'\perp \varphi_2'$ contains a hyperbolic subplane, but $\varphi_1'\perp \varphi_2' = \beta_1[1,\alpha_1] \perp \beta_2[1,\alpha_2] \perp \Qf{1,1} \isom \beta_1[1,\alpha_1] \perp \beta_2[1,\alpha_2] \perp \Qf{1,0} = \Alb' \perp \Qf{0}$, so $i_W(\varphi_1'\perp \varphi_2') > 0$ if and only if $\Alb'$ is isotropic, as in the proof of \Tref{twoforms}.
\end{proof}

By taking $n = 2$ in \Cref{7.5}, we obtain the following curious fact:
\begin{cor}
Let $\QA{\beta}{\alpha_i}$ be quaternion algebras. If for every $i,j$ the quaternion algebras $\QA{\beta}{\alpha_i}$ and $\QA{\beta}{\alpha_j}$ share an inseparable quadratic subfield, then
$$\MPf{\alpha_1,\beta} \perp \cdots \perp \MPf{\alpha_s,\beta} \sim \MPf{\alpha_1\cdots\alpha_s,\beta}$$
in $I_q^2F$.
\end{cor}

Recall that $\Alb = \beta_1[1,\alpha_1] \perp \beta_2[1,\alpha_2] \perp [1,\alpha_1+\alpha_2]$ is the Albert form of $Q_1 \tensor Q_2$ (see \cite[Section~16]{BOI}).
\begin{prop}\label{Albert}
The following conditions are equivalent.
\begin{enumerate}
\item $Q_1,Q_2$ have a common separable maximal subfield.
%\item $Q_1,Q_2$ have a common maximal subfield.
%\item $\ind(Q_1 \tensor Q_2) \leq 2$.
\item $\varphi_1,\varphi_2$ are right-linked.
\item The form $\Alb = \beta_1[1,\alpha_1] \perp \beta_2[1,\alpha_2] \perp [1,\alpha_1+\alpha_2]$ is isotropic.
\end{enumerate}
% For a condition that forces $\frak{a}$ to be anisotropic, see \Lref{monice2+} below.
\end{prop}
\begin{proof}
Again the equivalence $(1) \Leftrightarrow (2)$ is trivial. The Albert form of $Q_1 \tensor Q_2$ is isotropic if and only if $Q_1,Q_2$ have a common maximal subfield, see \cite{BOI}, Theorem~16.5 and Example~16.4.
\end{proof}

Notice that $\Alb' \sub \Alb$, which proves once more that left-linkage implies right-linkage. We also note that if $\varphi_1,\varphi_2$ are right-linked, then we may assume $\alpha_1 = \alpha_2 = \alpha$, and the non-trivial part $\beta_1[1,\alpha_1] \perp \beta_2[1,\alpha_2]$ of the Albert form $\Alb$ is similar to the Pfister form $\Alb_0 = \MPf{\beta_1\beta_2,\alpha}$, which is a factor of $\MPf{\beta_1,\beta_2,\alpha}$, whose Witt class is $\Sigma_{\set{\varphi_1,\varphi_2}}$. Now \Pref{6.3} and \Tref{twoforms} give two conditions for $\varphi_1$ and $\varphi_2$ to be left-linked, namely that $\beta_1\Alb_0 \perp \Qf{1}$ is isotropic and that $\MPf{\beta_1,\beta_2,\alpha}$ is isotropic, respectively. Indeed, the two conditions are equivalent because $\beta_1\Alb_0 \perp \Qf{1}$ is a Pfister neighbor of $\MPf{\beta_1,\beta_2,\alpha}$.

\begin{lem}\label{five_and_six}
Let $\Alb$ be any anisotropic Albert form over a field $F$ of characteristic~$2$. Any $5$-dimensional anisotropic quadratic form $\varphi$ over $F$ of the form $\varphi = [\alpha,\beta]\perp [\gamma,\delta] \perp \langle \epsilon \rangle$ remains anisotropic over $F(\Alb)$.
\end{lem}
\begin{proof}
If $\varphi$ is not a Pfister neighbor, then $\varphi_{F(\Alb)}$ is anisotropic by
\cite[Theorem 1.2 (2)]{Laghribi:2002}.
If $\varphi$ is a Pfister neighbor then it must be a Pfister neighbor of some anisotropic $3$-fold Pfister form $\phi$.
By the first paragraph of \cite[Section 1.2]{Laghribi:2002}, if $\varphi_{F(\Alb)}$ is isotropic then $\Alb$ is dominated by $\phi$, and since both $\Alb$ and $\phi$ represent $1$, this forces $\phi \isom \Alb \perp [a,b]$ for some $a,b \in F$ because $\Alb$ is nonsingular.
Recall \cite[Section~13]{EKM} that the Arf invariant $\Delta \co I_q^1 F \ra F/\wp(F)$, defined by $\Delta(\perp [a_i,b_i]) = \sum a_i b_i$ (where $\wp(F) = \set{\alpha^2+\alpha \suchthat \alpha \in F}$), has $\ker(\Delta) = I_q^2F$. Since $\Alb,\varphi \in I_q^2F$, we have that $\Delta([a,b]) = \Delta(\Alb)+\Delta([a,b]) = \Delta(\varphi) = 0$, so $[a,b] \isom \HQ$, contrary to the assumption that $\varphi$ is anisotropic.
%% This argument is correct as well, but the one above is simpler (according to the referee, at least).
%On one hand, the Clifford algebra $C(\phi)$ is isomorphic to $C(\Alb) \otimes_F C([a,b])$. On the other hand, since $\phi \in I^3(F)$, its Clifford algebra is matrices, and therefore $C(\Alb) \sim C([a,b])$ in the Brauer group. But this cannot happen, since $C(\Alb)$ has index $4$ and $C([a,b]) \isom  \QA{ab}{a}$ has index~$2$.
\end{proof}

\begin{cor}\label{six and down}
Assume $\Alb$ is an anisotropic Albert form over a field $F$ of characteristic~$2$. Then any proper subform of $\Alb$ remains anisotropic over $F({\Alb})$.
\end{cor}

\begin{prop}\label{Adam}
If $\varphi_1,\varphi_2$ are left-linked over the function field $F(\Alb)$, then they are right-linked over $F$.
\end{prop}
\begin{proof}
Otherwise, $\Alb$ is anisotropic over $F$ by \Pref{Albert}, and $\Alb'$ will remain anisotropic over $F(\Alb)$ by \Cref{six and down}, contrary to \Pref{6.3}.
\end{proof}

\subsection{The invariant $\Sigma$ for $2$-fold Pfister forms}

Let $S = \set{\varphi_1,\dots,\varphi_s}$ be a tight set of quadratic $2$-fold Pfister forms over the field $F$ of characteristic~$2$. Denote the corresponding quaternion algebras by $Q_1,\dots,Q_s$. By throwing out elements, we may assume the group generated by $S$ in $I_q^2F/I_q^3F \isom \Brp[2](F)$ has order $2^s$. For $S' \sub S$, let $Q_{S'}$ denote the quaternion algebra similar to $\bigotimes_{Q \in S'} Q$ (which exists by tightness). Let $\varphi_{S'}$ denote the associated norm form, so by definition $\Sigma = \Sigma_S = \sum_{\emptyset \neq S' \sub S} \varphi_{S'}$ in $I^2_qF$.
Each of the summands is a $2$-fold Pfister form, so $\dim(\perp_{\emptyset \neq S' \sub S} \varphi_{S'}) = 4(2^s-1)$.
\begin{prop}\label{updim}
$\dim(\an{\Sigma}) \leq 2^{s+1}$.
\end{prop}
\begin{proof}
Indeed, each of the $2^s-1$ summands has a subform of the form $[1,*]$, and $[1,t]\perp [1,t'] \sim [1,t+t']$, so $\dim(\an{\Sigma}) \leq 2 (2^s-1) + 2 = 2^{s+1}$. %Recall that $\Sigma \in I^3_qF$.
\end{proof}

\begin{prop}\label{1ll+}
Let $S = \set{\varphi_1,\dots,\varphi_s}$ be a tight set. Assume that some pair of forms in the group generated by $S$ is left-linked. Then $\dim(\an{\Sigma}) < 2^{s+1}$.
\end{prop}
\begin{proof}
By assumption there are distinct nonempty subsets $S',S'' \sub S$ such that $\varphi_{S'}$ and $\varphi_{S''}$ are left-linked. Obviously $Q_{S'\Delta S''} \sim Q_{S'} \tensor Q_{S''}$ (where $\Delta$ is the symmetric difference), but moreover $\varphi_{S' \Delta S''} = \varphi_{S'} + \varphi_{S''}$ in the Witt module by \Pref{2RL}. Canceling these three summands, $\Sigma$ is a sum of $2^s-4$ Pfister forms, so by the argument above, $\dim(\an{\Sigma}) \leq 2(2^s-4)+2 = 2^{s+1}-6 < 2^{s+1}$.
\end{proof}

\Tref{maincor++} shows that if $S$ is right-linked and pairwise left-linked, then it is strongly tight, so in particular $\Sigma = 0$. We can prove more:
\begin{cor}\label{1ll+?}
Assume $S = \set{\varphi_1,\dots,\varphi_s}$ is right-linked, and that some pair of forms in the group generated by $S$ is left-linked. Then $\Sigma_S = 0$.
\end{cor}
\begin{proof}
By \Pref{1ll+} $\dim(\an{\Sigma}) < 2^{s+1}$, but by \Pref{sep}, $\Sigma$ is represented by an $(s+1)$-fold Pfister form, which must therefore be hyperbolic.
\end{proof}

\section{Strong tightness does not imply linkage}\label{stnolink}

Fix $n \geq 2$, and let $k$ be any field of characteristic~$2$. Let $E = k(\alpha_0,\dots,\alpha_n)$ be the function field in $n+1$ algebraically independent variables over~$k$.
\begin{thm}\label{main8}
There is a strongly tight set of $n+1$ quadratic $n$-fold Pfister forms over $E$ with no common quadratic of bilinear $1$-fold Pfister subform (in particular the set is neither left- nor right-linked).
\end{thm}

After developing an additive notation for quadratic Pfister forms, we present the forms $\psi_i$ in \eq{formain8} below, and prove that $S = \set{\psi_0,\dots,\psi_n}$ is strongly tight.  The proof that the $\psi_i$ are not linked is given in \Sref{higherintersect} using valuations.

%\subsection{Jacobson notation for quadratic Pfister forms}

\subsection{Additive notation for quadratic Pfister forms}

Following Jacobson's additive and symmetric notation for quaternion algebras in characteristic~$2$, we denote $$((\alpha,\beta)) = \MPf{\alpha,\alpha\beta}$$ where $\alpha \in \mul{F}$ and $\beta \in F$. These symbols obviously generate $I_q^2F$ (\cite{Arason06} gives a presentation of $I_q^2F$ in these terms). For the sake of completeness, we set $((0,\beta))$, for any $\beta \in F$, to be equal to the hyperbolic $2$-fold Pfister form.

%We begin with some basic properties of the symbols for $n =2$:
\begin{prop}\label{rels}
The symbol $((\cdot,\cdot))$ is a biadditive alternating form in~$I_q^2$.
\forget
% Namely:
\begin{enumerate}
\item $((\beta,\alpha)) = ((\alpha,\beta))$;
\item $((\alpha,\beta))+((\alpha,\beta')) = ((\alpha,\beta+\beta'))$;
\item $((\alpha,\alpha)) = 0$.
\end{enumerate}
\forgotten
\end{prop}
\begin{proof}
%By \cite[Lemma~15.1]{EKM}
%
Additivity in the right slot follows from
\begin{eqnarray*}
((\alpha,\beta))+((\alpha,\beta')) & = & \MPf{\alpha,\alpha\beta}+\MPf{\alpha,\alpha\beta'} \\
& = & \MPf{\alpha,\alpha(\beta+\beta')} \\
& = & ((\alpha,\beta+\beta'))
 \end{eqnarray*}
 when $\alpha \neq 0$, and trivially holds for $\alpha = 0$.
The symbol is symmetric because $((\beta,\alpha)) - ((\alpha,\beta)) = \MPf{\beta,\alpha\beta} -\MPf{\alpha,\alpha\beta} = \Qf{\beta,\alpha}\MPf{\alpha\beta} = \Qf{\beta,\alpha}[1,\alpha\beta] = [\beta,\alpha]\perp [\alpha,\beta] = 0$ when $\alpha,\beta \neq 0$; and $((\beta,0)) = \MPf{\beta,0} = \HQ\perp \HQ = ((0,\beta))$. This shows additivity in the left slot as well.

Finally, $((\alpha,\alpha)) = \MPf{\alpha,\alpha^2} = \MPf{\alpha,\alpha} = 0\in I_q^2F$, so the symbol is alternating.
%Finally, $((\alpha,\alpha)) = \MPf{\alpha,\alpha^2} = \Qf{1,\alpha}[1,\alpha^2] = [1,\alpha^2]\perp [\alpha,\alpha]$ is isotropic, as both 2-dimensional forms represent $\alpha^2$
\end{proof}

%\subsection{Generalized Jacobson symbols}\label{sec:genJ}

%Our fields still have characteristic~$2$. In \Sref{quaternionCE} we presented a strongly tight set of $2$-fold Pfister forms, of size $s = 3$, which was neither left- nor right-linked.
%In this section we generalize Jacobson's notation to Pfister forms, and for any $n$, present a set of $n$-fold Pfister forms, of size $s = n+1$, which is strongly tight, and does not seem to be (generically) linked.

More generally,
we denote
$$((\alpha_1,\dots,\alpha_{n-1},\beta)) = \MPf{\alpha_1,\dots,\alpha_{n-1},\alpha_1\cdots\alpha_{n-1}\beta},$$
where $\alpha_1,\dots,\alpha_{n-1} \in \mul{F}$ and $\beta \in F$, and again set $((\alpha_1,\dots,\alpha_{n-1},\beta))$ to be the hyperbolic $n$-fold Pfister form when $\alpha_1\cdots\alpha_{n-1} =0$. Since $\MPf{0} \isom \HQ$, $((\alpha_1,\dots,\alpha_{n-1},\beta)) = 0$ if one of the entries is zero. %
Clearly, every quadratic $n$-fold Pfister form has this presentation.

\begin{thm}
The function $((\cdot,\dots,\cdot)) \co F\times \cdots \times F \rightarrow I_q^nF$ is multi-additive and alternating.
%((`determinant-like notation'))
\end{thm}
\begin{proof}
For $n = 1$, namely $((\alpha)) = \MPf{\alpha}$, there is nothing to prove. For $n = 2$  we proved the relations in \Pref{rels}. So we may assume $n > 2$.
The symbol is clearly symmetric in the first $n-1$ entries. To show that the symbol is symmetric, it suffices to check for symmetry of the final two entries. Indeed, when $\alpha_1,\dots,\alpha_{n-1},\beta \neq 0$, set $\alpha = \alpha_1\cdots\alpha_{n-2}$ and $b = \Pf{\alpha_1,\dots,\alpha_{n-2}}$. Since $\alpha$ is a value of $b$, it is a similarity factor, so $b[\alpha_{n-1},\alpha\beta] = b\alpha[\alpha_{n-1},\alpha\beta] = b[\alpha\alpha_{n-1},\beta]$; therefore
\begin{eqnarray*}
&& ((\alpha_1,\dots,\alpha_{n-2},\alpha_{n-1},\beta)) - ((\alpha_1,\dots,\alpha_{n-2},\beta,\alpha_{n-1})) \\
&    = & \MPf{\alpha_1,\dots,\alpha_{n-2},\alpha_{n-1},\alpha_1\cdots\alpha_{n-1}\beta} - \MPf{\alpha_1,\dots,\alpha_{n-2},\beta,\alpha_1\cdots\alpha_{n-1}\beta} \\
%& = & b \Qf{\alpha_{n-1},-\beta}\MPf{\alpha \alpha_{n-1}\beta} \\
& = & b (\Qf{\alpha_{n-1}}[1,\alpha \alpha_{n-1}\beta] - \Qf{\beta}[1,\alpha\alpha_{n-1}\beta]) \\
& = & b ([\alpha_{n-1},\alpha \beta] - [\beta,\alpha\alpha_{n-1}]) = 0;
\end{eqnarray*}

\forget
\begin{eqnarray*}
&& ((\alpha_1,\dots,\alpha_{n-2},\alpha_{n-1},\beta)) - ((\alpha_1,\dots,\alpha_{n-2},\beta,\alpha_{n-1})) \\
&    = & \MPf{\alpha_1,\dots,\alpha_{n-2},\alpha_{n-1},\alpha_1\cdots\alpha_{n-1}\beta} - \MPf{\alpha_1,\dots,\alpha_{n-2},\beta,\alpha_1\cdots\alpha_{n-1}\beta} \\
& = & \Qf{\alpha_{n-1},-\beta}\MPf{\alpha_1,\dots,\alpha_{n-2},\alpha_1\cdots\alpha_{n-1}\beta} \\
& = & \Pf{\alpha_1,\dots,\alpha_{n-2}}\Qf{\alpha_{n-1},-\beta}\MPf{\alpha_1\cdots\alpha_{n-1}\beta} \\
& = & \Pf{\alpha_1,\dots,\alpha_{n-2}}([\alpha_{n-1},\alpha_1\cdots\alpha_{n-2}\beta] + [\beta,\alpha_1\cdots\alpha_{n-1}]) \\
& = & \Pf{\alpha_1,\dots,\alpha_{n-2}}([\alpha_{n-1},\alpha\beta] + [\beta,\alpha\alpha_{n-1}]) \\
& = & \sum_{S\sub \set{1,\dots,n-2}} ([\alpha_S\alpha_{n-1},\alpha_{S^c}\beta] + [\alpha_S\beta,\alpha_{S^c}\alpha_{n-1}]) \\
& = & \sum [\alpha_S\alpha_{n-1},\alpha_{S^c}\beta] + \sum [\alpha_{S^c}\beta,\alpha_{S}\alpha_{n-1}] = 0;
%& = & \Qf{-\beta}\Pf{\alpha_{n-1}\beta^{-1}}\MPf{\alpha_1,\dots,\alpha_{n-2},\alpha_1\cdots\alpha_{n-1}\beta} \\
%& = & \Qf{-\beta}\MPf{\alpha_1,\dots,\alpha_{n-2},\alpha_{n-1}\beta,\alpha_1\cdots\alpha_{n-1}\beta} \\
\end{eqnarray*}
\forgotten
and if $\alpha_1\cdots \alpha_{n-1} \beta = 0$ then $((\alpha_1,\dots,\alpha_{n-2},\alpha_{n-1},\beta))$ and $((\alpha_1,\dots,\alpha_{n-2},\beta,\alpha_{n-1}))$ are both hyperbolic. Since the symbol is additive in the final entry, symmetry implies multi-additivity.

It remains to prove that the form is alternating, and since $n>2$ it suffices to notice that $\Pf{\alpha,\alpha} = 0$, which is the case because $\Qf{\alpha,\alpha} \sub \Pf{\alpha,\alpha}$ and $\Qf{\alpha,\alpha}$ is isotropic.
\end{proof}

We now prove the first claim of \Tref{main8}. We take the forms
\begin{equation}\label{formain8}
\psi_i = ((\alpha_0,\dots,\alpha_{i-1},\alpha_{i+1},\dots,\alpha_{n})),
\end{equation}
where $i = 0,\dots,n$. To show that $\set{\psi_0,\dots,\psi_n}$ is strongly tight,
%\begin{proof}[Proof that $\set{\psi_0,\dots,\psi_n}$ of \Tref{main8} is strongly tight]
we need to show that the sum of every subset of the generators $\psi_0,\dots,\psi_{n+1}$ is represented in the Witt module by a single Pfister form. Let $I$ be a nonempty subset of the index set $\set{0,\dots,n}$. Choose $i_0 \in I$. Consider the $n$-fold Pfister form $\psi' = ((\cdots))$, whose entries are the $\alpha_i$ for $i \not \in I$, and the sums $\alpha_{i_0}+\alpha_i$ for $i \in I \setminus\set{i_0}$. For example if $n = 5$ and $I = \set{0,1,3}$, we may choose $i_0 = 0$, and then $\psi' = ((\alpha_2,\alpha_4,\alpha_0+\alpha_1,\alpha_0+\alpha_3))$. By additivity and the fact that when $\alpha_{i_0}$ appears at least twice the form cancels, we see that $\psi' \sim  \sum_{i \in I} \psi_i$.
%Consider the wedge product $(\wedge_{i \not \in S'} \alpha_{i}) \wedge (\wedge_{i \in S', i \neq 0} (\alpha_0+\alpha_i))$:
%\end{proof}

\section{Valuations and Pfister forms}\label{higherintersect}

In this section we complete the proof of \Tref{main8} by showing that the forms $\psi_0,\dots,\psi_n$ of \Eq{formain8} do not have a $2$-dimensional common subform.

We begin with a useful lemma on monomial values.
Let $F$ be a field of characteristic~$2$, with a valuation $\nu \co F \ra \Gamma$, where $\Gamma$ is a totally ordered abelian group.
%Let $\Gamma^+ = \set{\gamma \in \Gamma \suchthat \gamma > 0}$ denote the positive semigroup.
For a quadratic form $\varphi$ defined on an $F$-vector space $W$, we let $\bar{\nu}(D(\varphi)) \sub \Gamma/2\Gamma$ denote the set of values under $\nu$, modulo $2$, of the elements $\varphi(w) \in \mul{F}$, ranging over the anisotropic vectors $w \in W$.

Fix $k \geq 1$. Consider the $2$-dimensional space $W_1 = Fe_0+Fe_1$, and inductively set $W_j = W_1 \tensor W_{j-1}$, where $j = 2,\dots,k$. For each $j$, $W_j$ has the natural monomial basis composed of the vectors $e_{\bf{i}} = e_{i_j} \tensor \cdots \tensor e_{i_1}$ for the $2^j$ index functions ${\bf{i}} \in \set{0,1}^j$.

Fix $\beta_1,\dots,\beta_k \in \mul{F}$. For $1 \leq j \leq k$, we define the forms $\varphi_j$ on $W_j$ by induction: for $j = 1$ we take $\varphi_1(w_0 e_0+ w_1 e_1) = w_0^2+w_0w_1+\beta_1 w_1^2$ ($w_0,w_1 \in F$), and for $j \geq 2$ we set $\varphi_j(e_0 \tensor w_0 + e_1 \tensor w_1) = \varphi_{j-1}(w_0) + \beta_j \varphi_{j-1}(w_1)$ ($w_0,w_1 \in W_{j-1})$. Thus $\varphi_j  \isom \MPf{\beta_j,\dots,\beta_1}$.

%Let $e_0,\dots,e_{2^k-1}$ denote the standard basis of the vector space $W = F^{(2^k)}$. Consider the $k$-fold Pfister form $\varphi  = \MPf{\beta_k,\dots,\beta_1}$ defined on $W$, where $\beta_1,\dots,\beta_k \in \mul{F}$.

\begin{lem}\label{monice}
Assume $\nu(\beta_1) < 0$ and the images of ${\nu(\beta_1)},{\nu(\beta_2)},\dots,{\nu(\beta_k)}$ in the quotient group $\Gamma/2\Gamma$ are linearly independent (over $\F_2$). Let $\varphi = \MPf{\beta_k,\dots,\beta_1}$. Then:
\begin{enumerate}
\item For every nonzero $w = \sum_{{\bf{i}}\in \set{0,1}^k}  w_{\bf{i}} e_{\bf{i}} \in W$ (where $w_{\bf{i}} \in F$), $\nu(\varphi(w)) = \nu(\varphi(w_{\bf{i}} e_{\bf{i}}))$ for some ${\bf{i}}$ for which $w_{\bf{i}} \neq 0$;
\item $\varphi$ is anisotropic;
\item $\bar{\nu}(D(\varphi))$ is the subspace of $\Gamma/2\Gamma$ spanned by $\nu(\beta_1),\dots,\nu(\beta_k)$.
\end{enumerate}
\end{lem}
\begin{proof}
It suffices to prove the first claim, which we do by induction. For $k = 1$, $\varphi_1(w_0e_0+w_1e_1) = w_0^2+w_0w_1+\beta_1 w_1^2$. If $w_0w_1= 0$ the first claim is immediate, so let $t = w_1^{-1}w_0$ and write $\varphi_1(w_0e_0+w_1e_1) = w_1^2(t^2+t+\beta_1)$. Let $\delta = \nu(t)$. If $0 \leq \delta$ then clearly $\nu(\beta_1) < 0 \leq \nu(t) \leq \nu(t^2)$, so $\beta_1$ is the monomial with least value in $t^2+t+\beta_1$ and $\nu(\varphi_1(w_0e_0+w_1e_1)) = \nu(\varphi_1(w_1e_1))$. Otherwise $\delta < 0$, so $\nu(t^2)< \nu(t)$, but $\nu(\beta_1) \neq \nu(t^2) \in 2\Gamma$ and the  minimum of $\nu(t^2)$, $\nu(t)$ and $\nu(\beta_1)$ is obtained only once, so $\nu(\varphi_1(w_0e_0+w_1e_1)) = \nu(\varphi_1(w_1e_1))$ or $\nu(\varphi_1(w_0e_0+w_1e_1)) = \nu(w_1^2t^2) = \nu(w_0^2) = \nu(\varphi_1(w_0e_0))$. %It immediately follows that $\varphi_1$ is anisotropic, and that $\bar{\nu}(D(\varphi_1)) = \set{\bar{0},\bar{\nu}(\beta_1)}$.

Assume the claim holds for $\varphi_{k-1} = \MPf{\beta_{k-1},\dots,\beta_1}$. Recall that for $w_0,w_1 \in W_{k-1}$ and $w = e_0 \tensor w_0 + e_1 \tensor w_1$, we have that $\varphi_k(w) = \varphi_{k-1}(w_0) + \beta_k \varphi_{k-1}(w_1)$. If $\varphi_0(w_0) = 0$ or $\varphi(w_1) = 0$, then since $\varphi_{k-1}$ is isotropic, we have that $w = e_1 \tensor w_1$ or $w = e_0\tensor w_0$, respectively, and we are done by induction. Otherwise the induction hypothesis shows that $\bar{\nu}(\varphi_{k-1}(w_0)), \bar{\nu}(\varphi_{k-1})(w_1)$ are distinct by linear independence, so $\bar{\nu}(\varphi_{k}(w))$ equals one of them.
\end{proof}

%\begin{cor}\label{monice0}
%Under the assumptions of \Lref{monice}, $\bar{\nu}(D(\varphi))$ is the subspace of $\Gamma/2\Gamma$ spanned by $\nu(\beta_1),\dots,\nu(\beta_k)$.
%\end{cor}

\begin{cor}\label{monice+}
Under the assumptions of \Lref{monice}, for every $2$-dimensional subspace $U \sub W$, $\bar{\nu}(D(\varphi|_U)) \neq 0$.
% there are elements $w\in U$ for which $\nu(\varphi(w)) \not \in 2 \Gamma$.
\end{cor}
\begin{proof}
Since $U$ is 2-dimensional, it contains a nonzero element $w$ in which the coefficient of $e_{\bf{0}} = e_0 \tensor \cdots \tensor e_0$ is zero.
Then by the previous lemma, $\nu(\varphi(w))=\nu(\varphi(w_{\bf{t}} e_{\bf{t}}))=\nu(w_{\bf{t}}^2) + \nu(e_{\bf{t}}) \in \nu(e_{\bf{t}}) + 2\Gamma$ for some ${\bf{t}} \neq {\bf{0}}$. But $\nu(e_{\bf{t}})$ is a nonempty partial sum of $\nu(\beta_1),\dots,\nu(\beta_k)$ which are linearly independent, so it is not in~$2\Gamma$.
\end{proof}

To complete the proof of \Tref{main8}, it remains to show that the $n$-fold Pfister forms $\psi_i$ of \Eq{formain8} do not share a common two-dimensional subform.
\begin{proof}[Completion of the proof of \Tref{main8}]
\sloppy We endow $E = k(\alpha_0,\dots,\alpha_n)$ with the $(\alpha_0^{-1},\dots,\alpha_n^{-1})$-adic valuation whose value set is
$\Gamma = \Z^{n+1}$, with some total ordering. Let $\pi_i \co \Gamma \ra \Z/2\Z$ be the projection on the $i$th component, modulo $2$.
%(or extend scalars to $k\db{\alpha_0^{-1}}\cdots\db{\alpha_n^{-1}}$).
For each $i$, the form
$$\psi_i = \MPf{\alpha_0,\dots,\widehat{\alpha_i},\dots,\alpha_{n-2},\alpha_0\cdots\widehat{\alpha_i}\cdots\alpha_{n-1}}$$
satisfies the conditions of \Lref{monice}, so by \Lref{monice}.3, $\nu(D(\psi_i)) = \ker(\pi_i)$. Therefore, $\bigcap \nu(D(\psi_i)) = \bigcap \ker \pi_i = (2\Z)^{n+1}$.  Let $U \sub W$ be a two-dimensional subspace, and let $\rho_i$ be the restriction of $\psi_i$ to $U$. If we assume all the $\rho_i$ are isometric to a form $\rho$, then $D(\rho) \sub D(\psi_i)$ forces $\nu(D(\rho)) \sub (2\Z)^{n+1}$, which contradicts \Cref{monice+}.
\end{proof}

This proves the second claim in \Tref{main8}.

\section{Triplets of quaternion algebras}\label{quaternionCE}

In this section we consider triplets of quaternion algebras. We show how the invariant can be used to prove the `additivity' of left-linkage from the right-linkage of a triplet, and provide an example where the invariant of a tight triplet is nonzero in $I_q^3F/I_q^4F$.

\subsection{Pairwise right-linked triplets}

%We view right-linkage as a weak condition, and left-linkage as a strong one. Indeed, the invariant of a triplet $S$ vanishes if it is left-linked.

Let $Q_i = \QA{\alpha_i}{\beta_i}$, $i = 1,2,3$, be three quaternion algebras over a field $F$ of characteristic~$2$, and $\varphi_i = \MPf{\beta_i,\alpha_i}$ the respective norm forms. In order for the invariant $\Sigma_{\varphi_1,\varphi_2,\varphi_3}$ to be defined, the triplet must be tight. In this case, there are quaternion algebras $Q_{ij}$ and $Q_{123}$ such that $Q_{ij} \sim Q_i \tensor Q_j$ and $Q_{123} \sim Q_1 \tensor Q_2 \tensor Q_3$. By definition the invariant is
\begin{equation}\label{S3}
\Sigma = \varphi_1 \perp \varphi_2 \perp \varphi_3 \perp \varphi_{12} \perp \varphi_{13} \perp \varphi_{23} \perp \varphi_{123},
\end{equation}
where each of the summands is the norm form of the respective algebra, namely a $2$-fold Pfister form.
% It follows that the dimension of the quadratic form in the right-hand side of \eq{S3} is $28$. However by
Recall that by \Pref{updim}, $\dim(\an{\Sigma}) \leq 16$.

\begin{rem}
Assume $\varphi_1,\varphi_2,\varphi_3$ are pairwise right-linked (without assuming tightness).

The algebras $Q_{ij}$ are defined, but $Q_{123}$ may not be. The triplet is tight if and only if $\ind(Q_1 \tensor Q_2 \tensor Q_3) \leq 2$, or equivalently $\ind(Q_i \tensor Q_{jk}) \leq 2$ (where $\set{i,j,k} = \set{1,2,3}$). This is equivalent to the Albert form $\Alb_i$ of $Q_i \tensor Q_{jk}$ being isotropic (see \Pref{Albert}).
\end{rem}

\forget
\begin{prop}\label{1ll}
Assume $Q_1,Q_2,Q_3$ is right-linked and that the pair $Q_1,Q_2$ is left-linked. Then $\Sigma = 0$.
\end{prop}
\begin{proof}
By assumption $\varphi_{12} = \varphi_1 + \varphi_2$ in the Witt module. Therefore,
$$\Sigma = \varphi_3 \perp \varphi_{13} \perp \varphi_{23} \perp \varphi_{123},$$
so by the argument above, $\dim(\an{\Sigma}) \leq 4 \cdot 2 + 2 \leq 10$. But $\Sigma$ is a $4$-fold Pfister form by \Pref{sep}, which proves that $\Sigma = 0$.
\end{proof}
(\Tref{maincor++} shows that if $Q_1,Q_2,Q_3$ is right-linked and pairwise left-linked, then it is strongly tight, so in particular $\Sigma = 0$).
((Stronger version: a right-linked set of quaternions, with one pair of generators which is left-linked, has $\Sigma = 0$, by the argument of \Pref{1ll})).
\forgotten

\forget % Done above
\begin{rem}
Assume $Q_1,Q_2,Q_3$ is tight, so in particular pairwise right-linked. Assume furthermore that $Q_1,Q_2$ are left-linked. It follows that $\varphi_{12} = \varphi_1 + \varphi_2$ in the Witt module. Therefore,
$$\Sigma = \varphi_3 \perp \varphi_{13} \perp \varphi_{23} \perp \varphi_{123},$$
so by the argument above, $\dim(\an{\Sigma}) \leq 4 \cdot 2 + 2 \leq 10$. If we assume further that $Q_1,Q_2,Q_3$ is right-linked, then $\Sigma$ is a $4$-fold Pfister form by \Pref{sep}, which proves that in this case $\Sigma = 0$.
\end{rem}
\forgotten

\begin{lem}\label{lem}
Assume $\varphi_1,\varphi_2,\varphi_3$ is a tight triplet of $2$-fold Pfister forms.
If the pairs $\varphi_1,\varphi_2$ and $\varphi_1,\varphi_3$ are left-linked, then $\Sigma_{\varphi_1,\varphi_2,\varphi_3} = \Sigma_{\varphi_1,\varphi_{23}}$.
\end{lem}
\begin{proof}
By assumption $\varphi_1 + \varphi_2 = \varphi_{12}$ and $\varphi_1 + \varphi_3 = \varphi_{13}$, so by \Eq{S3} we have that $\Sigma_{\varphi_1,\varphi_2,\varphi_3} = \varphi_1  \perp \varphi_{23} \perp \varphi_{123} = \Sigma_{\varphi_1,\varphi_{23}}$.
\end{proof}

\begin{prop}\label{8.7}
Assume $\varphi_1,\varphi_2,\varphi_3$ is right-linked, and that the pairs $\varphi_1,\varphi_2$ and $\varphi_1,\varphi_3$ are left-linked. Then $\varphi_1,\varphi_{23}$ are left-linked as well.
\end{prop}
\begin{proof}
By \Lref{lem} and \Cref{1ll+?}, $\Sigma_{\varphi_1,\varphi_{23}} = \Sigma_{\varphi_1,\varphi_2,\varphi_3} = 0$, so the claim follows from \Pref{2RL}.
\end{proof}

\forget % Done above
\begin{rem}
Assume $\varphi_1,\varphi_2,\varphi_3$ is tight. Assume furthermore that $\varphi_1,\varphi_2$ are left-linked and that $\varphi_1,\varphi_3$ are left-linked as well. Now $\varphi_1 + \varphi_3 = \varphi_{13}$, so as before,
$$\Sigma = \varphi_1  \perp \varphi_{23} \perp \varphi_{123}.$$
By the argument same $\dim(\an{\Sigma}) \leq 3 \cdot 2 + 2 \leq 8$. If we assume further that $\varphi_1,\varphi_2,\varphi_3$ is right-linked, then $\Sigma$ is a $4$-fold Pfister form by \Pref{sep}, which proves that in this case $\Sigma = \Sigma_{\varphi_1,\varphi_2,\varphi_3} = 0$. But we also have that $\Sigma_{\varphi_1,\varphi_{23}} = \Sigma = 0$, which proves by \Pref{2RL} that $\varphi_1$ and $\varphi_{23}$ are left-linked.
\end{rem}
\forgotten

\subsection{A triplet with $\Sigma = 0$ and no common subfields}

Taking $n = 2$ in \Tref{main8} and phrasing the result in terms of quaternion algebras, we conclude:
\begin{cor}\label{main7Q}
There is a triplet of quaternion division algebras $Q_1,Q_2,Q_3$ over the field~$k(\alpha_0,\alpha_1,\alpha_2)$, with quaternion algebras $Q_{ij} \sim Q_i \tensor Q_j$ and $Q_{123} \sim Q_1 \tensor Q_2 \tensor Q_3$, such that every two algebras in $Q_1,Q_2,Q_3,Q_{12},Q_{13},Q_{23},Q_{123}$ have a common inseparable maximal subfield, but $Q_1,Q_2,Q_3$ have no common maximal subfield.
\end{cor}
Notice that in this case $\Sigma_{\varphi_1,\varphi_2,\varphi_3} = 0$, where $\varphi_i$ is the norm form of~$Q_i$.

\subsection{A triplet with nonzero invariant}

We give an example of a tight triplet of $2$-fold Pfister forms which cannot be right-linked because its invariant in $I^3E/I^4E$ is nonzero.

%Recall that if $\varphi_2,\varphi_3$ are right-linked quaternion algebras, then $\varphi_{23}$ is the unique quaternion algebra similar to $\varphi_{2} \tensor \varphi_3$.
\begin{prop}\label{X}
%Let $\varphi_1,\varphi_2,\varphi_3$ be quaternion algebras over a field $F$ of characteristic~$2$ such that:
Assume each of the pairs $\varphi_1,\varphi_2$ and $\varphi_1,\varphi_3$ is left-linked;
the pair $\varphi_2,\varphi_3$ is right-linked; but $\varphi_1$ is not right-linked to $\varphi_{23}$.
%\begin{enumerate}
%\item Each of the pairs $\varphi_1,\varphi_2$ and $\varphi_1,\varphi_3$ is left-linked.
%\item The pair $\varphi_2,\varphi_3$ is right-linked.
%\item $\varphi_1$ is not right-linked to $\varphi_{23}$.
%\end{enumerate}

Then, letting $\Alb$ be the Albert form of $Q_1 \tensor Q_{23}$ (which is anisotropic by \Pref{Albert}) and taking $E = F(\Alb)$,
%\begin{enumerate}
%\item $\varphi_1,\varphi_2,\varphi_3$ is not tight over $F$.
%\item
the set $S = \set{\varphi_1,\varphi_2,\varphi_3}$ is tight over $E = F(\Alb)$, but $\Sigma_S \not \in I_q^4E$.
%\item In particular $\varphi_1,\varphi_2,\varphi_3$ are not right-linked over $E$.
%\end{enumerate}
\end{prop}

\begin{proof}
We work over $E$, so the triplet is tight. Denote $\Sigma = \Sigma_S$. By \Lref{lem}, $\Sigma =  \Sigma_{\varphi_1,\varphi_{23}}$ so $\dim(\an{\Sigma}) \leq 8$.
By assumption $\varphi_1,\varphi_{23}$ are not right-linked over $F$, so by \Pref{Adam} $\varphi_1,\varphi_{23}$ are not left-linked over $E$, and therefore $\Sigma \not \in I_q^4E$.
%If $\Sigma \in I_q^4E$, it follows that $\Sigma = 0$, which implies that $\varphi_1,\varphi_{23}$ are left-linked over $E$, and by \Lref{Adam} it follows that $\varphi_1,\varphi_{23}$ are right-linked over $F$, contrary to assumption.
\end{proof}

We conclude with an example which realizes the conditions of \Pref{X}. We need the following variant of \Lref{monice}:
\begin{lem}\label{monice2} % See generalization in Further ideas.
Let $K$ be a field of characteristic~$2$ with a discrete valuation $\nu \co K \ra \Gamma$, and $\alpha,\beta,\gamma \in K$. Assume $\nu(\beta) = \nu(\gamma) < 0$ and the images of $\nu(\alpha), \nu(\beta)$ are linearly independent in $\Gamma/2\Gamma$.
Then $\alpha[1,\beta]\perp [1,\gamma]$ is anisotropic.
\end{lem}
\begin{proof}
Since $\nu(\beta) = \nu(\gamma)  \neq 0$ in $\Gamma/2\Gamma$, if $d$ is a nonzero value of $\MPf{\beta}$ or of $\MPf{\gamma}$, then $\nu(d) \in \set{0,\nu(\beta)}$, where we consider values modulo $2\Gamma$. Therefore, a nonzero value $d$ of $\alpha\MPf{\beta}$ has $\nu(d) \in \set{\nu(\alpha),\nu(\beta)+\nu(\alpha)}$, which is disjoint from the set of possible values for $\MPf{\gamma}$.
\end{proof}

\begin{exmpl}[Realization of \Pref{X}]
Let $k$ be a field of characteristic $2$. Let $F=k(\alpha,\beta,\gamma)$ be the function field in three algebraically independent variables over $k$.
Let $\phi_1=\MPf{\gamma,\alpha}$, $\phi_2=\MPf{\gamma, \beta}$ and $\phi_3=\MPf{\alpha\gamma, \beta}$.
Clearly $\phi_2$ is right-linked to $\phi_3$ and $\phi_1$ is left-linked to $\phi_2$. Since $\phi_1 \isom \MPf{\alpha \gamma,\alpha}$, $\phi_1$ is also left-linked to $\phi_3$.
Let $\phi_{23}$ denote the form $\MPf{\alpha,\beta}$ which is equivalent modulo $I_q^3F$ to $\phi_2 \perp \phi_3$.

We need to show that $\phi_1$ is not right-linked to $\phi_{23}$. By \Pref{Albert}, we need to show that the Albert form $\Alb = \gamma [1,\alpha] \perp \alpha [1,\beta] \perp [1,\alpha+\beta]$ is anisotropic.
Since $F=k(\alpha,\beta)(\gamma)$, $\Alb$ is anisotropic if and only if $[1,\alpha]$ and $\alpha [1,\beta] \perp [1,\alpha+\beta]$ are anisotropic over $k(\alpha,\beta)$. But $[1,\alpha]$ is clearly anisotropic, and $\alpha [1,\beta] \perp [1,\alpha+\beta]$ is anisotropic by \Lref{monice2} once we equip $K = k(\alpha,\beta)$ with the $(\alpha^{-1},\beta^{-1})$-valuation, with $\nu(\alpha) > \nu(\beta)$.
\end{exmpl}

\section*{Acknowledgements}
We thank Jean-Pierre Tignol, Adrian Wadsworth, and an anonymous referee for their comments on the manuscript.

% For iffurther

\section*{Bibliography}
\bibliographystyle{amsalpha}
\bibliography{bibfile}
\end{document}